\documentclass{amsart}
\usepackage[utf8]{inputenc}
\usepackage{amssymb}
\usepackage{color}
\usepackage{mathrsfs}

\title{Extensions between Verma modules for dihedral groups}
\author{Gurbir Dhillon and Visu Makam}
\newtheorem{theorem}{Theorem}[section]
\newtheorem{lemma}[theorem]{Lemma}
\newtheorem{corollary}[theorem]{Corollary}
\newtheorem{proposition}[theorem]{Proposition}
\theoremstyle{definition}
\newtheorem{definition}[theorem]{Definition}

\newtheorem{remark}[theorem]{Remark}

\newcommand{\Hom} {\operatorname{Hom}}
\newcommand{\gHom} {\operatorname{gHom}}
\newcommand{\End} {\operatorname{End}}
\newcommand{\Ext} {\operatorname{Ext}}
\newcommand{\modf}{\operatorname{mod^{fg}}}
\newcommand{\gmodf}{\operatorname{gmod^{fg}}}
\newcommand{\Z} {\mathbb Z}

\newcommand{\gExt} {\operatorname{gExt}}
\newcommand{\SB} {\mathcal{SB}}
\newcommand{\op} {\operatorname{op}}
\newcommand{\SM}{\mathcal{SM}}
\newcommand{\gr}{\operatorname{gr}}
\newcommand{\C} {\mathcal{C}}

\newcommand\iso{\xrightarrow{
   \,\smash{\raisebox{-0.65ex}{\ensuremath{\scriptstyle\sim}}}\,}}

\thanks{The first author was partially supported by NDSEG fellowship. The second author was supported by NSF grant DMS-1601229}

\begin{document}

\maketitle

\begin{abstract}
Computing the extensions between Verma modules is in general a very difficult problem. Using Soergel bimodules, one can construct a graded version of the principal block of Category $\mathcal{O}$ for any finite coxeter group. In this setting, we compute the extensions between Verma modules for dihedral groups.
\end{abstract}

\section{Introduction}
One way to construct the finite dimensional irreducible representations of a semisimple lie algebra is to construct them as the unique simple quotients of a Verma modules corresponding to dominant integral weights. Verma modules themselves are not finite dimensional, and the natural setting to study them is a category that is commonly referred to as Category $\mathcal{O}$. We are interested in computing the extensions between Verma modules. The Gabber-Joseph conjecture states that the dimensions of extensions of Verma modules are given by R-polynomials. This conjecture was however disproved by Boe in \cite{Boe}, and in \cite{Abe}, Abe studies how far the dimensions of the first extension groups differ from these polynomials. Very little is known about higher extension groups.

The principal block $\mathcal{O}_0$ of Category $\mathcal{O}$ (see Section~\ref{alt}) is a highest weight category, and the underlying poset is the Weyl group with partial order given by reverse bruhat order. The Verma modules in this principal block are the standard objects. Indeed, using Soergel's ideas, one can construct a category equivalent to $\mathcal{O}_0$ using only the coxeter presentation of the Weyl group. This construction (see Section~\ref{sb}), yields $\mathcal{O}_0$ as the category of finitely generated modules over an algebra. This algebra is in fact a graded algebra, and all the important modules (including Verma modules) admit graded lifts, allowing us to consider the category of finitely generated graded modules over this graded algebra as a graded version of $\mathcal{O}_0$. 
For two objects $M$ and $N$ in a graded abelian category, $\Ext^j(M,N)$ is graded, and we denote the $i^{th}$ graded part of the $j^{th}$ extension group by $\gExt^j_i(M,N)$. We state the main result of this paper. 

\begin{theorem} \label{main}
Let $(W,S)$ be a dihedral group. Let $x,y \in W$, and let $i,j \in \mathbb Z$ such that $l(x) - l(y) = 2j+i \geq 0$. Then we have 
$$\rm \dim \gExt^j_i(\Delta_x,\Delta_y) = \begin{cases} 2 & \text{if } j>0, i+j >0 \\ 1 & \text{if } j(i+j) = 0\end{cases}.$$

In all other cases, the dimension is zero.
\end{theorem}
We define a generating function, $e(x,y) = (\dim \gExt^j_i(\Delta_x,\Delta_y))q^jt^i$.

\begin{corollary}
Let $r = l(x) - l(y) \geq 0$. Then we have $$e(x,y) = t^r + 2qt^{r-2} + 2q^2t^{r-4} + \dots + 2q^{r-1}t^{-(r-2)} + q^rt^{-r}.$$
\end{corollary}

\section{Category $\mathcal{O}$} \label{catO}
In this section, we will briefly recall Category $\mathcal{O}$ associated to a semisimple lie algebra $\mathfrak{g}$, and describe how the principal block $\mathcal{O}_0$ may be viewed as a category of modules over a graded algebra. In the subsequent sections, we will then use this description to define a Category $\mathcal{O}_0$ for any coxeter group. 

Let $\mathfrak{g}$ denote a semisimple lie algebra, and $\mathfrak{h}$ a choice of Cartan subalgebra. Let $\Phi \subseteq \mathfrak{h}^*$ denote its root system, and let $W$ denote the Weyl group. A choice of a generic hyperplane in $\mathfrak{h}^*$ allows us to split the root system $\Phi = \Phi_+ \cup \Phi_-$ into positive roots and negative roots. Let $\mathfrak{g} = \mathfrak{n}^- \oplus \mathfrak{h} \oplus \mathfrak{n}^+$ be the corresponding Cartan decomposition, and let $\mathfrak{b} = \mathfrak{h} \oplus \mathfrak{n}^+$ denote the corresponding Borel subalgebra. We will denote by $\mathcal{U}\mathfrak{g}$ (resp. $\mathcal{U} \mathfrak{b}$), the universal enveloping algebra of $\mathfrak{g}$ (resp. $\mathfrak{b}$). 

\begin{definition}
Category $\mathcal{O}$ is the full subcategory of $\mathcal{U}\mathfrak{g}$ modules whose objects satisfy

\begin{enumerate}
\item $M$ is $\mathfrak{h}$-semisimple;
 \item $M$ is finitely generated as a $\mathcal{U}(\mathfrak{g})$ module;  
\item $M$ locally $\mathfrak{n}^+$ finite.
\end{enumerate}
\end{definition}

For each $\lambda \in \mathfrak{h}^*$, we have a Verma module $\Delta_{\lambda} := \mathcal{U}\mathfrak{g} \otimes_{\mathcal{U}\mathfrak{b}} \mathbb C_{\lambda}$, whose unique simple quotient $L_{\lambda}$ is the simple module of highest weight $\lambda$. The projective cover of $L_{\lambda}$ is denoted $P_{\lambda}$, and its injective hull is denoted $I_{\lambda}$. We define the shifted action of $W$ on $\mathfrak{h}^*$ as $w \cdot \lambda = w(\lambda + \rho) - \rho$.

Category $\mathcal{O}$ splits into blocks, i.e., $$\mathcal{O} = \bigoplus_{\mathcal{X} \in \mathfrak{h}/(W\cdot)}\mathcal{O}_\mathcal{X}.$$

The principal block $\mathcal{O}_0 = \mathcal{O}_{(W\cdot 0)}$ is the block associated to the orbit of $0 \in \mathfrak{h}^*$ under the shifted action. For each $w \in W$, we define $L_w := L_{w\cdot 0}$. These are all the simple modules in $\mathcal{O}_0$. Similarly, we define $\Delta_w:= \Delta_{w\cdot 0}$, and $P_w = P_{w\cdot 0}$, etc. There is a duality functor $\mathbb{D}$ on $\mathcal{O}$, and we write $\mathbb{D}(\Delta_\lambda) := \nabla_\lambda.$

For any ring $A$, we denote by $A-{\rm mod}^{\rm fg}$, the category of finitely generated left modules over $A$. The direct sum of indecomposable projectives $P = \bigoplus_{w \in W} P_w$  is a projective generator in $\mathcal{O}_0$, so by Morita theory, we have an equivalence of categories:
\begin{align*} \mathcal{O}_0 & \iso \End(P)^{\op}-{\rm mod}^{\rm fg} \\ M & \longmapsto  \Hom(P,M) \end{align*}

\section{Soergel Bimodules} \label{sb}
The results on Soergel bimodules that we use can all be found in \cite{EW,Soergel}.

\subsection{Grading conventions} Before we proceed further, it will be important to fix our grading conventions to avoid any confusion.
For a graded module $M = \oplus_{i \in \Z} M_i$, we define its graded shift by $M(n)$ by $M(n)_i = M_{i+n}$. If $M$ and $N$ are graded modules, then let $\gHom(M,N) = \Hom_0(M,N) = \{f \in \Hom(M,N)\ |\ f(M_k) \subseteq N_k\}$. If we define $\gHom_i(M,N) =\{f \in \Hom(M,N) \ |\ f(M_k) \subseteq N_{i+k}\ \forall k\}$. In particular, we have $\gHom_i(M,N) = \gHom (M(-i),N) = \gHom (M,N(i))$.

Let $A$ be a graded ring. By $A-\gmodf$, we mean finitely generated graded modules over $A$ where the morphims between two modules $M$ and $N$ are $\gHom(M,N)$.

\subsection{Soergel bimodules} Let $(W,S)$ be a finite coxeter system and let $\mathfrak{h}^*$ denote its geometric representation. Let $R = \rm Sym(\mathfrak{h})$ denote the symmetric algebra over $\mathfrak{h}$, or equivalently the ring of polynomial functions on $\mathfrak{h}^*$. The ring $R$ is a polynomial ring, and hence naturally $\Z$-graded. We wish to consider $R$ as an evenly graded algebra. In other words, we adopt the convention that the linear polynomials are in degree $2$, i.e., $R_2 = \mathfrak{h}$. 
 
For $s\in S$, let $B_s$ denote the R-bimodule $R \otimes_{R^s} R(1)$. We define the category $\SB$ of Soergel bimodules is the smallest additive Karoubian category containing all the $B_s$, and is stable under arbitrary shifts. In other words, it is the full subcategory of $R$-bimodules whose objects are $R$-bimodules that are isomorphic to direct sums of graded shifts of direct summands of bimodules of the form $B_s \otimes B_t \otimes \dots \otimes B_u$ for $s,t,\dots,u \in I$. The indecomposable objects in $\SB$ are indexed by the Weyl group $W$. Indeed, for any reduced expression $w = st\dots u$, $B_w$ can be characterized as the unique summand of $B_s \otimes B_t \otimes \dots \otimes B_u$ that has not appeared for any smaller expression. 

The category $\SM$ of Soergel modules is the category of right $R$-modules obtained by tensoring the objects in $\SB$ on the left by $(R/R_+)$ over $R$. In other words, the indecomposable objects in $\SM$ (up to grading shifts) are $\overline{B_w} := (R/R_+) \otimes_R B_w$. Indeed, it was shown by Soergel that $\Hom_{\SM}(\overline{B_x},\overline{B_y}) = \Hom_{\SB}(B_x,B_y)$.

\subsection{Category $\mathcal{O}$} If $(W,S)$ is the Weyl group of a semisimple lie algebra $\mathfrak{g}$, then we can define the principal block $\mathcal{O}_0$ of $\mathfrak{g}$ in terms of Soergel modules. We define $\overline{B} = \bigoplus_{w \in W} \overline{B_w}$, and let $\mathcal{A}_{W} = \End(\overline{B})$. Soergel shows that we have an isomorphism $A_W = \End(P)$. As we have seen, $\mathcal{O}_0$ is equivalent to $\End(P)^{\op}-\modf$, which is $A_W^{\op}-\modf$. The ring $A_W$ is naturally graded, and all the important modules (Verma, projective, simple etc) all have graded lifts. Thus we define $\mathcal{O}_0^{\gr} := A_W^{\op}-\gmodf$ as the graded version of Category $\mathcal{O}_0$. However, Soergel modules are naturally defined for any coxeter group, and we will use this to give a definition of $\mathcal{O}_0^{\gr}$ for any finite coxeter group.

\begin{definition}
We define $\mathcal{O}_0^{\rm gr}$ to be the category of finitely generated graded modules over $A_W^{\op}$.
\end{definition}

\subsection{Jantzen filtration}
In this section, let $\mathfrak{g} \supset \mathfrak{b} \supset \mathfrak{h}$ be a semisimple lie algebra with a choice of borel subalgebra and a choice of maximal torus. Let $(W,S)$ be the corresponding Weyl group. There is a filtration on the Verma module $\Delta_x = \Delta_x^0 \supseteq \Delta_x^1 \supseteq \Delta_x^2 \dots$, with $\Delta_x^i = 0$ for large enough $i$ called the Jantzen filtration. In fact, this coincides with the socle filtration and the radical filtration, and moreover, the layers $\Delta_x^k/\Delta_x^{k+1}$ are semisimple, see \cite[Chapter~8]{Humphreys}.

The counterpart of the Jantzen filtration in geometry is the weight filtration. In order to see this filtration, we must give another construction (the Koszul dual construction), and we do so in the next section.

\section{Another construction of $\mathcal{O}_0^{\gr}$} \label{alt}
We give another construction of $\mathcal{O}_0^{\gr}$ as the heart of a perverse $t$-structure on the bounded homotopy category of Soergel modules where the Jantzen filtration is far more apparent.  

Let $(W,S)$ denote a finite coxeter group, and let $\SM$ denote the corresponding category of Soergel modules. Consider the bounded homotopy category of Soergel modules $K^b(\SM)$. In what follows, $(i)$ denotes the grading shift by $i$ and $[i]$ denotes the cohomological shift by $[i]$. We define the tate twist $\left<1 \right> := (-1)[1]$. That Soergel's conjecture is true for the geometric realization allows one to define a certain perverse $t$-structure, and its corresponding heart gives an abelian subcategory $\C$ of $K^b(\SM)$. The category $C$ consists of bounded complexes $M^\bullet = \dots \rightarrow M^0 \rightarrow M^1 \dots$ ($M^0$ is in cohomological degree $0$), such that $M^i = \oplus_{w \in W} B_w(i)^{m(w,i)}$. The simple modules in $\mathcal{C}$ are $B_w\left<i\right>$ for $w \in W$ and $i \in \Z$.

In \cite{Mak}, the category $\mathcal{C}$ is proven to be a graded highest weight category, whose underlying poset is $W$ with the reverse bruhat order. In particular, for each $w \in W$, we have a simple module $L_x$, a projective cover of $L_x$ denoted $P_x$, a standard object (Verma module) $\Delta_x$ etc. Let $\mathcal{P} = \bigoplus_{x \in W} P_x$. Consider the ring $\End(\mathcal{P})$. In \cite{Mak}, it is shown that $\End(\mathcal{P})$ is Koszul dual to $A_W$, and that $\C = \End(\mathcal{P})^{\op}-\modf$. In \cite{Mak2}, it is shown that $A_W$ is Koszul self dual. Putting the two results together, we get that $\mathcal{O}^{\rm gr} \iso \C$. Under this isomorphism, the grading shift $(i)$ gets sent to $\left<i\right>$.

\begin{remark}
The standard objects (Verma modules) in this construction of $\mathcal{O}_0^{\gr}$ correspond to Rouquier complexes, see \cite{Mak2}. We will only need rouquier complexes at one point during our computation, at which point we will refer to an explicit description of the rouquier complexes for dihedral groups, which can be found, for e.g. in \cite{AMRW}. 
\end{remark}

\begin{remark}
In this description, the Jantzen filtration is far more apparent, and is given by successive truncations of the complex. 
\end{remark}

\section{Proof of main result}
Let $(W,S)$ denote a dihedral group. In this section, by $\mathcal{O}_0^{\gr}$, we will mean the category defined in Definition~\ref{defO}. It is sometimes advantageous to consider the equivalent category $\C$ constructed in Section~\ref{alt}, and we will do so freely. We know that $\mathcal{O}_0^{\gr} = \C$ is a graded highest weight category for the poset $W$ with reverse bruhat order, and hence for each $x \in W$, we have the objects $L_x,P_x,\Delta_x$ (simple, projective, standard). We assume that the reader is familiar with graded highest weight categories, and we refer to \cite{AR} for those who aren't. In \cite{Sau}, an explicit projective resolution for $\Delta_x$ is given.

\begin{proposition} \label{proj.res} 
Let $x \in W$, and let $P^j := \bigoplus\limits_{\substack{y\leq x \\ l(x) - l(y) = j}} P_y (-j)$. Then $$P^{l(x)} \rightarrow P^{l(x) - 1} \rightarrow \dots \rightarrow P^0 \rightarrow 0$$ is a projective resolution of $\Delta_x$. 

\end{proposition}

\begin{lemma}
Let $\Delta_y = \Delta_y^0 \supseteq \Delta_y^1 \supseteq \Delta_y^2 \dots$ be the weight filtration. Then 
$$\Delta_y^i/ \Delta_y^{i+1} = \begin{cases} \bigoplus\limits_{l(x) = l(y) + i} \hspace{-10pt} L_x(-i),  & i \geq 1 \\ \hspace{15pt}  L_y, & i = 0 \end{cases}
$$

\end{lemma}

\begin{proof}
This is clear from the explicit description of Rouquier complexes for dihedral groups, see \cite[Chapter~9]{AMRW}. 
\end{proof}

\begin{proposition}
We have $$\dim \gHom(\Delta_x,\Delta_y(i)) = \begin{cases} 1 & \text{if } x \geq y \text{ and } i = l(x) - l(y)  \\ 0 & \text{otherwise} \end{cases}.$$
\end{proposition}

\begin{proof}
We have a surjection $P_x \twoheadrightarrow \Delta_x$, giving us an injection $\gHom(\Delta_x,\Delta_y(i)) \hookrightarrow \gHom(P_x,\Delta_y(i))$. We have $\dim \gHom(P_x,\Delta_y(i)) = [\Delta_y:L_x(i)] = 1$ if and only if $i = l(x) - l(y)$. Hence, we have 
$$\dim \gHom(P_x,\Delta_y(i)) = \begin{cases} 1 & \text{if } x \leq y \text{ and } i = l(x) - l(y)  \\ 0 & \text{otherwise} \end{cases}.$$

Hence it suffices to show that when $x \geq y$ and $i = l(x) - l(y)$, we have a degree $0$ graded map from $\Delta_x$ to $\Delta_y(i)$. Consider the truncated full subcategory generated by all $L_z$ for $l(z) \geq l(x) = l(y) + i$. From the axioms of a graded highest weight category, it follows that since $z$ is maximal in this poset (under reverse bruhat order), that $\Delta_x$ is a projective cover of $L_x$ in this truncated subcategory. Now, consider the weight filtration for $\Delta_y$ and consider the submodule $\Delta_y^i$. It is clear from the previous lemma that $\Delta_y^i$ is in this truncated subcategory, and further that $ [\Delta_y^i(i):L_x] = [\Delta_y^i:L_x(-i)] = 1$. Since $\Delta_x$ is the projective cover of $L_x$ in this truncated subcategory, the canonical projection $\Delta_x \rightarrow L_x$ lifts to a nonzero map $\Delta_x \rightarrow \Delta_y^i(i) \hookrightarrow \Delta_y(i)$.
\end{proof}

\begin{remark}
We know from the representation theory that any nonzero map between Verma modules must be injective. It can be deduced easily from the previous proposition that nonzero maps between Verma modules are injective for dihedral groups as well. The argument relies mostly on the fact that Kazhdan-Lusztig polynomials are trivial for dihedral groups. 

However, there are two finite coxeter groups $H_3$ and $H_4$, which do not arise as the Weyl group of a semisimple lie algebra nor are they dihedral. It is a natural question to ask whether the statement is true in these cases as well. In fact, rouquier complexes can be defined for any coxeter group, and it would be interesting to know whether such a statement holds in this generality.
\end{remark}

The proof of the above proposition tells us that the injection $\gHom(\Delta_x,\Delta_y(i)) \hookrightarrow \gHom(P_x,\Delta_y(i))$ is indeed an isomorphism. We record this as a corollary for later use. 

\begin{corollary} \label{factor}
Any map $\varphi:P_x \rightarrow \Delta_y(i)$ factors through $\Delta_x$.
\end{corollary}

\begin{proposition}
Let $x,y,z \in W$ with $l(z) < l(x)$ and let $\phi:P_z \rightarrow P_x(i)$ be a map. Let $f_{x,y}:P_x(i) \rightarrow \Delta_y(j)$ be a nonzero map. Then we have $f_{x,y} \circ \phi = 0$.
\end{proposition}

\begin{proof} 
Let ${\mathcal{I}}$ denote the image of $f_{x,y} \circ \phi$. Since $f_{x,y}$ factors through $\Delta_x(i)$, and $[\Delta_x(i):L_z(j)] = 0$ for all $j$, we have that $[\mathcal{I}:L_z(j)] = 0$ for all $j$. 

On the other hand, any nonzero map from $P_z \rightarrow \Delta_y(j)$ factors through the projection to $\Delta_z$ Hence, the image is isomorphic to a quotient of $\Delta_z$, and any quotient of $\Delta_z$ has $L_z$ as the unique simple quotient. Thus, if $f_{x,y} \circ \phi$ was a nonzero map, we would have $[\mathcal{I}:L_z] = 1$.
\end{proof}

\begin{corollary} \label{diff.zero}
Let $P^\bullet$ denote the projective resolution for $\Delta_x$ as in Proposition~\ref{proj.res}. Then we have $$\dim \gExt^j(\Delta_x,\Delta_y(i)) = \dim \gHom(P^j,\Delta_y(i)).$$
\end{corollary}

\begin{proof}
All the differentials in the complex $\gHom(P^\bullet,\Delta_y(i))$ are $0$ by the above proposition.
\end{proof}

\begin{proof} [Proof of Theorem~\ref{main}]
The theorem follows from the Corollary~\ref{diff.zero}.
\end{proof}

\subsection*{Acknowledgements}
The authors would like to thank Ben Elias for helpful discussions. The second author would like to thank Thomas Lam for interesting discussions as well as suggesting the problem.

\end{document}